\documentclass[reqno, a4paper]{amsart}

\usepackage{amsfonts}
\usepackage{amsmath}
\usepackage{amssymb}
\usepackage{amsthm}
\usepackage[backend=biber, giveninits=true, maxbibnames=99, sortcites=true]{biblatex}
\usepackage{enumitem}
\usepackage{mathtools}
\usepackage[hidelinks, pdfusetitle]{hyperref}
\usepackage[capitalise]{cleveref}
\usepackage{crossreftools}

\DeclarePairedDelimiterX{\ha}[2]{\{}{\}}{#1 \,\delimsize|\, #2}
\DeclarePairedDelimiter{\zj}{(}{)}

\DeclarePairedDelimiterX{\sk}[2]{\langle}{\rangle}{#1, #2}

\DeclareMathOperator*{\argmin}{arg\,min}
\DeclareMathOperator{\prox}{prox}
\DeclareMathOperator{\id}{id}

\DeclareMathOperator{\Lip}{Lip}

\newcommand{\R}{\mathbb R}
\newcommand{\C}{\mathbb C}
\newcommand{\Bbar}{\overline B}
\newcommand{\one}{\mathbf 1}

\theoremstyle{plain}
\newtheorem{Thm}{Theorem}
\newtheorem{Prop}[Thm]{Proposition}
\newtheorem{Cor}[Thm]{Corollary}
\newtheorem{Lemma}[Thm]{Lemma}

\theoremstyle{definition}
\newtheorem{Def}[Thm]{Definition}
\newtheorem{Rem}[Thm]{Remark}
\newtheorem{Question}[Thm]{Question}

\numberwithin{Thm}{section}

\crefname{Prop}{Proposition}{Propositions}
\crefname{Thm}{Theorem}{Theorems}
\crefname{Lemma}{Lemma}{Lemmas}
\crefname{equation}{}{}

\setlist[enumerate]{label=(\arabic*)}

\renewbibmacro{in:}{}
\DeclareFieldFormat{pages}{#1}
\DeclareFieldFormat[article]{volume}{\mkbibbold{#1}}
\DeclareFieldFormat[article, inbook]{title}{\mkbibitalic{#1}}
\DeclareFieldFormat{booktitle}{#1}
\DeclareFieldFormat{journaltitle}{#1\isdot}

\title{Solution of Laczkovich's Lipschitz mapping problem}
\author{Attila Gáspár}
\address{Institute of Mathematics, ELTE Eötvös Loránd University, Pázmány Péter s. 1/C, 1117 Budapest, Hungary \newline\indent
HUN-REN Alfréd Rényi Institute of Mathematics, Reáltanoda u. 13-15, 1053 Budapest, Hungary}
\email{gsprati99@gmail.com}

\begin{document}

\begin{abstract}
    In 1991, Laczkovich asked whether every measurable set of positive measure in $\R^d$ can be mapped onto a cube with a Lipschitz map. We show that the answer is positive for every $d$. Our results also give a positive answer for Kolmogorov's problem on contractions onto polygons with small measure loss in the case of Jordan measurable sets.
\end{abstract}

\maketitle

\section{Introduction}

In 1991, Laczkovich \cite{Laczkovich} asked the following question:
\begin{Question}[Laczkovich]
    Let $A \subseteq \R^d$ be a measurable set of positive Lebesgue measure. Is there a Lipschitz map that maps $A$ onto a cube?
\end{Question}
Equivalently, we may ask whether there is a Lipschitz map such that the image has nonempty interior.

For $d = 1$, an explicit construction of such a map can be easily given. In the case $d = 2$, a positive answer was first given by Preiss \cite{Preiss}. Using a combinatorial argument based on the Erdős--Szekeres theorem, Matoušek \cite{Matousek} later proved the following absolute constant version of the result:
\begin{Thm}[Matoušek]
    There exists a positive constant $c$ with the property that for every measurable set $A \subseteq \R^2$ with $\lambda(A) = 1$, there exists a 1-Lipschitz surjection $f : A \to [0, c]^2$.
\end{Thm}
\noindent Here, $\lambda(A)$ denotes the Lebesgue measure of $A$.

An alternative proof for the planar case was found by Jones using an earlier result of Uy \cite{Uy}:
\begin{Thm}[Uy]
    For every compact set $C \subseteq \C$ with $\lambda(C) > 0$, there is a non-constant Lipschitz map $f : \overline \C \to C$ that is holomorphic on $\overline \C \setminus C$.
\end{Thm}
\noindent Using a topological argument (see \cite{ACsP}), it can be shown that $f(C)$ has nonempty interior.

Laczkovich's question and related problems were also discussed in \cite{alberti2005structure,keleti2014hausdorff,mendel2013ultrametric,zindulka2020mapping,balka2024lipschitz,keleti1994peculiar,BEM}. However, until now, the case $d \ge 3$ was open. In this paper, we will prove that the answer is positive for every $d$:
\begin{Thm} \label{lipschitz_onto_ball}
    For every measurable set $A \subseteq \R^d$ with $\lambda(A) > 0$, there exists a 1-Lipschitz map $f : \R^d \to \R^d$ such that $f(A)$ has nonempty interior.
\end{Thm}

The main tool in the proof is the theory of \emph{generalized strips} introduced by the author in \cite{GStrip}. The definition uses some concepts from convex analysis. Denote by $\Gamma_0(\R^d)$ the set of lower semicontinuous convex functions $\R^d \to \R \cup \{\infty\}$ that are not infinite everywhere. Note that every convex function $\R^d \to \R$ is in $\Gamma_0(\R^d)$.

\begin{Def}
    The \emph{proximal map} of $f \in \Gamma_0(\R^d)$ is defined as
    \[\prox_f(x)=\argmin_{y\in \R^d} \zj*{f(y) + \frac{|x - y|^2}2},\]
    where $\argmin_{y} g(y)$ denotes the unique minimizer of $g$.
\end{Def}

The existence and uniqueness of the minimizer follow from the assumption $f \in \Gamma_0(\R^d)$ (see \cite[Proposition 12.15]{ConvAnal}).

We can now define generalized strips and the outer measure $\gamma$. We call a function $f:\R^d \to \R$ \emph{polyhedral} if it is the maximum of finitely many affine functions.

\begin{Def}
    Let $f:\R^d \to \R$ be a polyhedral function. The \emph{generalized strip associated to $f$} (denoted by $S(f)$) is the set of points $x\in \R^d$ such that $f$ is not differentiable at $\prox_f(x)$.

    A set $S\subseteq \R^d$ is called a \emph{generalized strip} if $S=S(f)$ for some polyhedral function $f$. The \emph{width} $w(S)$ of a generalized strip is defined as
    \[w(S)=2\inf\ha{\Lip f}{S = S(f)},\]
    where $\Lip f$ is the Lipschitz constant of $f$.
\end{Def}

\begin{Def}
    We define the outer measure $\gamma : \mathcal P(\R^d) \to [0, \infty]$ as
    \[\gamma(A) = \inf \ha*{\sum_{i=1}^\infty w(S_i)}{\text{$S_1, S_2, \dots$ are generalized strips such that $A \subseteq \bigcup_{i=1}^\infty S_i$}}.\]
\end{Def}

The following basic property of $\gamma$ was proved in \cite[Proposition 3.12]{GStrip}:
\begin{Prop} \label{gamma_null_impl}
    If $\gamma(A) = 0$ for some $A \subseteq \R^d$, then $\lambda(A) = 0$.
\end{Prop}

The key result used in the proof is the following equality:
\begin{Thm} \label{main}
    Let $C \subseteq \R^d$ be a compact set. Then
    \[\gamma(C) = \max_{g \in \Gamma_0(\R^d)} 2r(\prox_g(C)),\]
    where $r(A)$ denotes the supremum of radii of balls contained in $A$. Furthermore, this maximum is attained for some convex Lipschitz function $g : \R^d \to \R$.
\end{Thm}

From \cref{gamma_null_impl,main}, we can easily prove \cref{lipschitz_onto_ball}. First, by the regularity of the Lebesgue measure, it is enough to prove \cref{lipschitz_onto_ball} for a compact set $C$. If $\lambda(C) > 0$, then by \cref{gamma_null_impl}, we have $\gamma(C) > 0$. By \cref{main}, there exists some convex $g : \R^d \to \R$ such that $2r(\prox_g(C)) = \gamma(C) > 0$. Thus, $\prox_g(C)$ has nonempty interior. This completes the proof, since $\prox_g$ is a 1-Lipschitz map (see \cite[Proposition 12.28]{ConvAnal}).

Another application of \cref{main} is showing that the implication in \cref{gamma_null_impl} is actually an equivalence:
\begin{Thm} \label{gamma_null_equiv}
    For any $A \subseteq \R^d$, $\gamma(A) = 0$ if and only if $\lambda(A) = 0$.
\end{Thm}
The reverse implication follows from an upper bound for $\gamma$ that will be proved in \cref{gamma_lambda_ub}. We remark that \cref{gamma_null_equiv} answers all three questions posed in \cite{GStrip}.

We will also give a quantitative version of \cref{lipschitz_onto_ball}:
\begin{Thm} \label{lipschitz_quantitative}
    Let $C \subseteq B = \Bbar(0, M)$ be a compact set in $\R^d$ with $\lambda(C) > 0$. Let $R \in [0, M)$ such that $\lambda(B \setminus C) = \omega_d R^d$, where $\omega_d$ is the volume of the $d$-dimensional unit ball. Then there exists a 1-Lipschitz surjection from $C$ onto $\Bbar (0, M - R) \subseteq \R^d$.
\end{Thm}
As a corollary, we will obtain a generalization of Matoušek's absolute constant result:
\begin{Thm} \label{abs_const}
    For every $d$, there exists a positive constant $c$ with the property that for every measurable set $A \subseteq \R^d$ with $\lambda(A) = 1$, there exists a 1-Lipschitz surjection $f : A \to [0, c]^d$.
\end{Thm}

\cref{gamma_null_equiv} gives further corollaries for a question proposed by Kolmogorov (see \cite{KQuestion}):
\begin{Question}[Kolmogorov]
    Let $A\subseteq \R^2$ be a measurable set such that its Lebesgue measure $\lambda(A)$ is finite. Is it true that for every $\varepsilon > 0$, there exists a 1-Lipschitz map $f$ such that $f(A)$ is a polygon and $\lambda(f(A)) \ge \lambda(A) - \varepsilon$?
\end{Question}
\noindent While such a map does not always exist by a result of Balka, Elekes and Máthé \cite{BEM}, for compact $A$, the question is open. A positive result was given in \cite[Theorem 1.9]{GStrip} under the assumption that $A$ is a bounded set with $\gamma(\partial A) = 0$. Combined with \cref{gamma_null_equiv}, this gives the following corollary:
\begin{Cor}
    Let $A \subseteq \R^d$ be a Jordan measurable set. For every $\varepsilon > 0$, there exists a 1-Lipschitz map $f : \R^d \to \R^d$ such that $\lambda(f(A)) \ge \lambda(A) - \varepsilon$; furthermore, $f(A)$ is a polyhedron in the sense that its boundary can be covered by finitely many hyperplanes.
\end{Cor}
The above mentioned \cite[Theorem 1.9]{GStrip} also applies to a variant of Kolmogorov's question concerning maps with uniformly small displacement. Combined with \cref{gamma_null_equiv}, we obtain another corollary:
\begin{Cor}
    Let $A \subseteq \R^d$ be a compact Lebesgue null set. For every $\varepsilon > 0$, there exists a 1-Lipschitz map $f : \R^d \to \R^d$ such that $|f(x) - x| \le \varepsilon$ for every $x \in \R^d$, and $f(A)$ can be covered by finitely many hyperplanes.
\end{Cor}
\noindent This gives a positive answer to \cite[Question 4.2]{BEM}.

\section*{Use of artificial intelligence}

The proof of \cref{main} was obtained using ChatGPT on 26 October 2026. In particular, the proof of \cref{gamma_ub}, including the statement and proof of \cref{gamma_ub_lemma}, is entirely due to GPT-6 Astra. This proof was rewritten by the author to make it easier to understand. The proof of \cref{gamma_lambda_ub} was also written with the help of GPT-6 Sol. Additionally, GPT-6 Sol and Astra were used for verifying the correctness of the proofs, fixing inaccuracies and improving language.

The lower bound for $\gamma(C)$ in \cref{gamma_lb}, and the consequences of the equality in \cref{main} were already known by the author in early 2026; no AI tools were used for these results. We remark that while the lower bound is not necessary for proving \cref{lipschitz_onto_ball}, it was essential for coming up with the statement of \cref{main}. When asked to give lower bounds for the value of $\gamma$ for specific sets, GPT-6 Astra did not come up with this lower bound, though it could prove it given the statement.

\section{Proof of \cref{main}}

In the proof, we will use some basic facts from convex analysis.

\begin{Def}
    Let $f : \R^d \to \R \cup \{\infty\}$ be a function, and let $x\in \R^d$. The \emph{subdifferential} of $f$ at $x$ (denoted by $\partial f(x)$) is the set
    \[\partial f(x) = \ha[\big]{v \in \R^d}
    {\forall y \; f(y) \ge f(x) + \sk{v}{y - x}}.\]
\end{Def}

\begin{Prop}[{\cite[Example 20.3]{ConvAnal}}] \label{subdiff_mono}
    Let $f \in \Gamma_0(\R^d)$. If $v \in \partial f(x)$ and $u \in \partial f(y)$, then $\sk{v - u}{x - y} \ge 0$.
\end{Prop}

\begin{Prop}[{\cite[Proposition 16.44]{ConvAnal}}] \label{prox_subdiff}
    Let $f \in \Gamma_0(\R^d)$. For any $x, y\in \R^d$, $y = \prox_f(x)$ if and only if $x - y \in \partial f(y)$.
\end{Prop}

\begin{Prop}[{\cite[Proposition 24.8 (iii)]{ConvAnal}}] \label{prox_shift}
    Let $g \in \Gamma_0(\R^d)$ and $v \in \R^d$. Define the function $h(x) = g(x + v) + \sk{x}{v}$. Then $\prox_h = \prox_g - v$.
\end{Prop}

We will also use the following property of $\gamma$:
\begin{Lemma}[{\cite[Lemma 3.15]{GStrip}}] \label{gstrip_compact_cover}
    For a compact set $C\subseteq \R^d$,
    \[\gamma(C) = \inf\ha*{w(S)}{\text{$S\supseteq C$ is a generalized strip}}.\]
\end{Lemma}

We begin by proving the lower bound for $\gamma(C)$:
\begin{Prop} \label{gamma_lb}
    Let $g \in \Gamma_0(\R^d)$. For every compact set $C \subseteq \R^d$,
    \[\gamma(C) \ge 2 r(\prox_g(C)).\]
\end{Prop}
\begin{proof}
    It is enough to show that for every $r > 0$ and $v_0 \in \R^d$ such that $B(v_0, r) \subseteq \prox_g(C)$, the inequality $\gamma(C) \ge 2r$ holds. First, note that by \cref{prox_shift}, we may assume without loss of generality that $v_0 = 0$.

    By \cref{gstrip_compact_cover}, it is enough to show that for any polyhedral function $f : \R^d \to \R$ such that $C\subseteq S(f)$, we have $\Lip f \ge r$. Write $f(x) = \max_{i} (\sk{v_i}{x} + c_i)$, where none of the affine functions are redundant (in particular, this implies that the $v_i$ are distinct). Assuming for contradiction that $\Lip f < r$, we have $|v_i| < r$ for every $i$.
    
    We can choose $i$ such that $g(v_i) + c_i$ is maximal. Since $g$ is finite on $B(0, r)$, this maximum is finite. Among these, choose $i$ such that $|v_i|$ is maximal. Let $u_\varepsilon = (1 + \varepsilon) v_i$ for $\varepsilon > 0$. For sufficiently small $\varepsilon > 0$, we have $u_\varepsilon \in B(0, r)$. Choose an $x_\varepsilon \in C$ such that $\prox_g(x_\varepsilon) = u_\varepsilon$, and let $y_\varepsilon = x_\varepsilon - u_\varepsilon$. It follows from \cref{prox_subdiff} that $y_\varepsilon \in \partial g(u_\varepsilon)$.
    
    We claim that for every $\varepsilon > 0$ for which $x_\varepsilon$ is defined, there is a $j \ne i$ such that $v_j \in \partial f(x_{\varepsilon} - v_i)$. Assuming for contradiction that there is no such $j \ne i$, we have $\partial f(x_{\varepsilon} - v_i) = \{v_i\}$, since the subdifferential is the convex hull of some of the vectors $v_j$. It follows from \cref{prox_subdiff} that $\prox_f(x_{\varepsilon}) = x_{\varepsilon} - v_i$, therefore, $x_{\varepsilon} \notin S(f)$. This contradicts the assumption that $C \subseteq S(f)$.
    
    By the compactness of $C$, there is a sequence $\varepsilon_n \to 0$ such that $x_{\varepsilon_n} \to x$ for some $x \in C$. Let $y = x - v_i$, clearly $y_{\varepsilon_n} \to y$. Since the set $\ha{(z, v)}{z \in \partial g(v)}$ is closed (see \cite[Proposition 16.36]{ConvAnal}), we have $y \in \partial g(v_i)$. For every $n$, we also have a $j \ne i$ such that $v_j \in \partial f(x_{\varepsilon_n} - v_i)$. Since $j$ has finitely many possible values, we may assume that $j$ is the same for every $n$. Notice that $x_{\varepsilon_n} - v_i \to y$, hence $v_j \in \partial f(y)$. Since $f$ was written in a form where none of the affine functions is redundant, this implies that $f(y) = \sk{v_j}{y} + c_j$.

    Using the fact that $y \in \partial g(v_i)$, we can see that
    \begin{align*}
        g(v_j) + c_j &\ge g(v_i) + \sk{v_j - v_i}{y} + c_j \\
        &= g(v_i) - \sk{v_i}{y} + (\sk{v_j}{y} + c_j) \\
        &= g(v_i) - \sk{v_i}{y} + f(y) \\
        &\ge g(v_i) - \sk{v_i}{y} + (\sk{v_i}{y} + c_i) \\
        &= g(v_i) + c_i.
    \end{align*}
    Since $i$ maximizes $g(v_i) + c_i$, it follows that in all the above inequalities, we have equality. In particular, $g(v_j) = g(v_i) + \sk{v_j - v_i}{y}$, implying that $y \in \partial g(v_j)$, and $f(y) = \sk{v_i}{y} + c_i$, implying that $v_i \in \partial f(y)$. Furthermore, $g(v_j) + c_j = g(v_i) + c_i$, so by the maximality of $|v_i|$, we have $|v_j| \le |v_i|$.

    Since $y_{\varepsilon_n} \in \partial g(u_{\varepsilon_n})$ and $y \in \partial g(v_j)$, it follows from \cref{subdiff_mono} that
    $0 \le \sk{y_{\varepsilon_n} - y}{u_{\varepsilon_n} - v_j}$,
    or equivalently $\sk{y_{\varepsilon_n} - y}{v_j} \le \sk{y_{\varepsilon_n} - y}{u_{\varepsilon_n}}$.

    We can similarly use \cref{subdiff_mono} for $v_j \in \partial f(x_{\varepsilon_n} - v_i)$ and $v_i \in \partial f(y)$:
    \begin{align*}
        0 &\le \sk{v_j - v_i}{(x_{\varepsilon_n} - v_i) - y} \\
        &= \sk{v_j - v_i}{y_{\varepsilon_n} - y + u_{\varepsilon_n} -  v_i} \\
        &= \sk{v_j - v_i}{y_{\varepsilon_n} - y + \varepsilon_n v_i} \\
        &= \sk{v_j - v_i}{y_{\varepsilon_n} - y} + \varepsilon_n \sk{v_j - v_i}{v_i} \\
        &\le \sk{u_{\varepsilon_n} - v_i}{y_{\varepsilon_n} - y} + \varepsilon_n \sk{v_j - v_i}{v_i} \\
        &= \sk{\varepsilon_n v_i}{y_{\varepsilon_n} - y} + \varepsilon_n \sk{v_j - v_i}{v_i} \\
        &= \varepsilon_n \sk{v_i}{y_{\varepsilon_n} - y + v_j - v_i}.
    \end{align*}
    Dividing by $\varepsilon_n$ then taking the limit $n \to \infty$ yields the inequality $0 \le \sk{v_i}{v_j - v_i}$. It follows from the property $|v_j| \le |v_i|$ that
    \[|v_i - v_j|^2 = \zj[\big]{|v_j|^2 - |v_i|^2} - 2\sk{v_i}{v_j - v_i} \le 0,\]
    therefore $v_i = v_j$, contradicting the choice of $j$.
\end{proof}

\begin{Rem}
    The proof above can be modified to show that for $g \in \Gamma_0(\R^d)$ and a polyhedral function $f : \R^d \to \R$, we have
    \[r(\prox_g(S(f))) \le \Lip f,\]
    from which \cref{gamma_lb} follows by an application of \cref{gstrip_compact_cover}. The reason we prove the weaker form is that the non-compactness of $S(f)$ would make the limiting arguments and the calculations more complicated.
\end{Rem}

Next, we will show the upper bound for $\gamma(C)$. The proof relies on the following lemma:

\begin{Lemma} \label{gamma_ub_lemma}
    Let $C\subset\R^d$ be nonempty and compact, let $R>0$, and let $V\subseteq B(0,R)$ be a finite nonempty set. Set $M=\max_{x\in C} |x|$. If $\gamma(C) \ge 2R$, then there exists a polyhedral convex function $g$ with $\Lip g\leq M+R$ and $V\subseteq\prox_g(C)$.
\end{Lemma}
\begin{proof}
    Let $V = \{v_1, \dots, v_n\}$. Consider a map of the form
    \[f(y) =\max_i \zj[\big]{\sk{v_i}{y}+c_i}\]
    for $c = (c_1, \dots, c_n) \in \R^n$. It is clear that $\Lip f < R$, so by the assumption $\gamma(C) \ge 2R$, we have $C \nsubseteq S(f)$. Choose some $x \in C \setminus S(f)$, and let $y = \prox_f(x)$. Then $f$ is differentiable at $y$, so it is of the form $z \mapsto \sk{v_i}{z} + c_i$ in a neighborhood. For every $j$, we have
    \[\sk{v_j}{y}+c_j \le f(y) = \sk{v_i}{y} + c_i.\]
    It follows from \cref{prox_subdiff} that $y = x - v_i$. Substituting into the previous inequality gives the bound
    \[c_i \ge \sk{v_j - v_i}{x - v_i} + c_j.\]
    Let
    \[T_i(c) = \min_{x \in C} \max_j (\sk{v_j - v_i}{x - v_i} + c_j).\]
    The previous calculation shows that $c_i \ge T_i(c)$. We can also see that the case $j = i$ gives the bound $c_i \le T_i(c)$. Therefore, we have seen that for any $c \in \R^n$, there is some $i$ such that $T_i(c) = c_i$.

    Let $K = \max_{x \in C} \max_{i, j} |\sk{v_j - v_i}{x - v_i}|$. We can see that if $m(c) = \max_{i} c_i$, then $|T_i(c) - m(c)| \le K$ for every $i$. It follows that the map $\Phi(c) = T(c) - m(c) \one$ maps $[-K, K]^n$ into itself. We can check that $\Phi$ is continuous, so by Brouwer's fixed point theorem, it has some fixed point $c = \Phi(c)$. We know that for some $i$, $T_i(c) = c_i = T_i(c) - m(c)$, thus $m(c) = 0$, therefore $T(c) = c$.

    For each $i$, choose an $x_i \in C$ such that $c_i = T_i(c) = \max_{j}\zj{\sk{v_j - v_i}{x_i - v_i} + c_j}$. We have
    \[\sk{v_j - v_i}{x_i - v_i} \le c_i - c_j\]
    for every $i$ and $j$. Now define the function
    \[g_i(v) = \sk{v - v_i}{x_i - v_i} - c_i,\]
    and let $g = \max_i g_i$. We can see that for every $i$ and $j$,
    \[g_i(v_j) = \sk{v_j - v_i}{x_i - v_i} - c_i \le (c_i - c_j) - c_i = -c_j = g_j(v_j).\]
    It follows that $g(v_j) = g_j(v_j)$. Since $g_j$ has gradient $x_j - v_j$, we have $x_j - v_j \in \partial g(v_j)$, so by \cref{prox_subdiff},
    $\prox_g(x_j) = v_j$. We have thus shown that $V \subseteq \prox_g(C)$. Since $|x_i - v_i| \le M + R$, it is clear that $\Lip g\le M + R$, completing the proof.
\end{proof}

To finish the proof of \cref{main}, it remains to prove the following proposition:
\begin{Prop} \label{gamma_ub}
    For every nonempty compact set $C \subseteq \R^d$, there is a convex Lipschitz function $g : \R^d \to \R$ such that $\Bbar(0, \gamma(C) / 2) \subseteq \prox_g(C)$.
\end{Prop}
\begin{proof}
    The case $\gamma(C) = 0$ is trivial, so we may assume that $2R := \gamma(C) > 0$. Choose a sequence $\varepsilon_n \searrow 0$, and for each $n$, choose a finite $\varepsilon_n$-net $V_n$ of $B(0, R)$. By \cref{gamma_ub_lemma}, there is some $(R + M)$-Lipschitz convex function $g_n : \R^d \to \R$ such that $V_n \subseteq \prox_{g_n}(C)$, where $M = \max_{x \in C} |x|$.
    
    Adding a constant to a function does not change the proximal map, so we may assume that $g_n(0) = 0$ for every $n$. Since each $g_n$ is $(R + M)$-Lipschitz, we can apply the Arzelà--Ascoli theorem. Passing to a subsequence, we may assume $g_n$ converges locally uniformly to a convex $(R + M)$-Lipschitz function $g : \R^d \to \R$. It remains to show that $\Bbar(0, R) \subseteq \prox_g(C)$. Let $v \in \Bbar(0, R)$, and choose a sequence of points $v_n \in V_n$ converging to $v$. For each $n$, there is some $x_n \in C$ such that $\prox_{g_n}(x_n) = v_n$. By the compactness of $C$, passing to a further subsequence, we may assume that $x_n \to x \in C$.

    By \cref{prox_subdiff}, we have $x_n - v_n \in \partial g_n(v_n)$, that is,
    \[g_n(u) \ge g_n(v_n) + \sk{u - v_n}{x_n - v_n}\]
    for every $u \in \R^d$. Taking the limit $n \to \infty$ yields
    \[g(u) \ge g(v) + \sk{u - v}{x - v},\]
    or equivalently, $x - v \in \partial g(v)$. Applying \cref{prox_subdiff} again, this shows that $v = \prox_g(x)$, completing the proof.
\end{proof}

\Cref{gamma_lb,gamma_ub} together give the proof of \cref{main}.

\section{Bounds for \texorpdfstring{$\gamma$}{γ} and the proof of \cref{abs_const}}

In the next proposition, we will give an explicit upper bound for $\gamma$ in terms of the Lebesgue outer measure, from which the reverse direction of \cref{gamma_null_equiv} follows immediately. Recall that $\omega_d$ is the volume of the $d$-dimensional unit ball.

\begin{Prop} \label{gamma_lambda_ub}
    For any set $A \subseteq \R^d$,
    \[\gamma(A) \le 2\zj*{\frac{\lambda^*(A)}{\omega_d}}^{1/d},\]
    where $\lambda^*$ is the Lebesgue outer measure.
\end{Prop}
\begin{proof}
    First, consider the case when $A$ is compact. By \cref{main}, it is enough to show that if $g : \R^d \to \R$ is a convex function, then
    \[r(\prox_g(A)) \le \zj*{\frac{\lambda(A)}{\omega_d}}^{1/d},\]
    or equivalently,
    \[r(\prox_g(A))^d \omega_d \le \lambda(A).\]
    Since $\prox_g$ is 1-Lipschitz, we have $\lambda(\prox_g(A)) \le \lambda(A)$, from which the previous inequality follows immediately.

    Now let $A$ be an arbitrary set with $\lambda^*(A) < \infty$, and let $\varepsilon > 0$. We can take a cover $A \subseteq \bigcup_{n=0}^\infty C_n$, where each $C_n$ is compact and $\sum_{n=0}^\infty \lambda(C_n) \le \lambda^*(A) + \varepsilon$. By grouping the sets $C_n$ and replacing the sets in each group with their union in an appropriate way, we may assume that for $n \ne 0$,
    \[\lambda(C_n) \le \zj*{\frac{\varepsilon}{2^n}}^d \omega_d.\]
    Applying the bound $\lambda(C_0) \le \lambda^*(A) + \varepsilon$ and the previous bound for each $n \ne 0$ gives
    \begin{align*}
        \frac{\gamma(A)}{2} &\le \sum_{n=0}^\infty \frac{\gamma(C_n)}{2} \\
        &\le \sum_{n=0}^\infty \zj*{\frac{\lambda(C_n)}{\omega_d}}^{1/d} \\
        &= \zj*{\frac{\lambda(C_0)}{\omega_d}}^{1/d} + \sum_{n=1}^\infty \zj*{\frac{\lambda(C_n)}{\omega_d}}^{1/d} \\
        &\le \zj*{\frac{\lambda^*(A) + \varepsilon}{\omega_d}}^{1/d} + \sum_{n=1}^\infty \frac{\varepsilon}{2^n} \\
        &= \zj*{\frac{\lambda^*(A) + \varepsilon}{\omega_d}}^{1/d} + \varepsilon. \\
    \end{align*}
    Taking the limit $\varepsilon \to 0$ gives the desired bound.
\end{proof}

This upper bound for $\gamma$ also allows us to give a lower bound for $\gamma$ in terms of the Lebesgue outer measure.
\begin{Cor} \label{gamma_lambda_lb}
    Let $A \subseteq B = \Bbar(0, M)$, and let $R \in [0, M]$ such that $\lambda(B) - \lambda^*(A) = \omega_d R^d$. Then $\gamma(A) \ge 2(M - R)$.
\end{Cor}
\begin{proof}
    First, assume that $A$ is measurable. By \cref{gamma_lambda_ub}, we have
    \[\gamma(B \setminus A) \le 2\zj*{\frac{\lambda(B \setminus A)}{\omega_d}}^{1/d} = 2R.\]
    On the other hand, since $\prox_0 = \id$, it follows from \cref{gamma_lb} that
    \[\gamma(B) \ge 2r(\prox_0(B)) = 2r(B) = 2M.\]
    By the subadditivity of $\gamma$, we have
    \[\gamma(A) \ge \gamma(B) - \gamma(B \setminus A) \ge 2(M - R).\]
    
    If we do not assume that $A$ is measurable, then it is easy to see from the definition of $\gamma$ that we can choose some Borel set $A \subseteq A' \subseteq B$ such that $\gamma(A') = \gamma(A)$. We may further assume that $\lambda(A') = \lambda^*(A)$. The bound for $\gamma(A)$ then follows immediately from the bound for $\gamma(A')$.
\end{proof}

\begin{Rem}
    Though the proof of the implication $\gamma(A) = 0 \Rightarrow \lambda(A) = 0$ in \cite[Proposition 3.12]{GStrip} did not give a quantitative bound, the calculations there can be done more carefully to get an alternative proof of \cref{gamma_lambda_lb} that does not rely on \cref{main}.
\end{Rem}

Combined with \cref{gamma_ub}, this immediately gives the proof of \cref{lipschitz_quantitative}.

To prove \cref{abs_const}, we will use the following lemma from Matoušek's proof:
\begin{Lemma}[{\cite[Lemma 3.1]{Matousek}}] \label{matousek_lemma}
    For any $d$ and any $\varepsilon \in (0, 1)$, there exists a positive constant $c > 0$ such that for any measurable set $A \subseteq \R^d$ with $\lambda(A) = 1$, there is a 1-Lipschitz map $f : A \to Q = [0, c]^d$ such that $\lambda(f(A)) \ge (1 - \varepsilon) \lambda(Q)$.
\end{Lemma}
Even though this lemma was originally stated only for $d = 2$, the proof can be generalized to any $d$ in a straightforward way.

\begin{proof}[Proof of \cref{abs_const}]
    Let $c_0 > 0$ be the constant from \cref{matousek_lemma} for $\varepsilon = 1/2$. Then for any measurable $A \subseteq \R^d$ with $\lambda(A) = 1$, there is a 1-Lipschitz map $f_1 : A \to [0, c_0]^d$ such that $\lambda(f_1(A)) \ge c_0^d/2$. By the regularity of the Lebesgue measure, we can choose a compact set $C \subseteq f_1(A)$ such that $\lambda(C) \ge c_0^d/4$. Since $C \subseteq \Bbar(0, c_0 \sqrt d)$, we can apply \cref{lipschitz_quantitative} to get a 1-Lipschitz surjection $f_2 : C \to \Bbar(0, R)$, where $R \in (0, c_0 \sqrt d]$ and
    \[\omega_d\zj[\big]{c_0 \sqrt d - R}^d = \omega_d \zj[\big]{c_0 \sqrt d}^{d} - \lambda(C).\]
    We can also define the constant $R_0\in (0, c_0 \sqrt d]$ such that
    \[\omega_d\zj[\big]{c_0 \sqrt d - R_0}^d = \omega_d \zj[\big]{c_0 \sqrt d}^{d} - \frac{c_0^d}{4}.\]
    It follows from the bound $\lambda(C) \ge c_0^d/4$ that $R \ge R_0$. Therefore, $\Bbar(0, R) \supseteq [0, c]^d$, where $c = R_0 / \sqrt d$. Consequently, if $f_3 : \R^d \to [0, c]^d$ is an arbitrary 1-Lipschitz retraction, then $f_3(\Bbar(0, R)) = [0, c]^d$.
    
    By Kirszbraun's theorem, we can extend $f_2$ to a 1-Lipschitz map defined on $\R^d$ (alternatively, notice that the proof of \cref{lipschitz_quantitative} already gives a map on $\R^d$). We can now take the 1-Lipschitz map $f = f_3 \circ f_2 \circ f_1$, which maps $A$ onto $[0, c]^d$. This completes the proof, since $c$ depends only on $d$.
\end{proof}

\section*{Acknowledgements}
The author would like to thank his supervisors M. Elekes and T. Keleti for their guidance and many helpful discussions.

The author was supported by the National Research, Development and Innovation Office -- NKFIH, grants nos. 146922 and 152822. This material is based upon work partly supported by the Swedish Research Council under grant no. 2021-06594 while the author was in residence at Institut Mittag-Leffler in Djursholm, Sweden during the autumn of 2026.

\printbibliography

\end{document}